\documentclass{amsart}
\usepackage{amssymb}
\usepackage{amsmath}
\usepackage{graphicx}
\usepackage{enumerate}
\usepackage{color}
\newcommand{\G}{\mathcal{G}}

\newcommand{\R}{\mathbb{R}}

\newcommand{\eps}{\varepsilon}
\newcommand{\seq}[1]{\langle #1 \rangle}

\newtheorem{thm}{Theorem}[section]

\newtheorem{que}[thm]{Question}
\newtheorem{lem}[thm]{Lemma}

\newtheorem{prop}[thm]{Proposition}

\newtheorem*{thm-A}{Theorem~\ref{thm-A}}
\newtheorem*{thm-B}{Theorem~\ref{thm-B}}
\newtheorem*{thm-C}{Theorem~\ref{thm-C}}
\newtheorem*{cor-D}{Corollary~\ref{cor-D}}
\newtheorem*{thm-E}{Theorem~\ref{thm-E}}

\newcommand{\Orb}{\mathrm{Orb}}

\DeclareMathOperator{\Int}{int}

\DeclareMathOperator{\diam}{diam}

\numberwithin{equation}{section}
\begin{document}
%

\title{All minimal Cantor systems are slow
}

\author{Jan P. Boro\'nski}
\address[J. P. Boro\'nski]{AGH University of Science and Technology, Faculty of Applied
	Mathematics, al.
	Mickiewicza 30, 30-059 Krak\'ow, Poland}
\email{jboronski@wms.mat.agh.edu.pl}

\address{	-- and -- National Supercomputing Centre IT4Innovations, Division of the University of Ostrava,
	Institute for Research and Applications of Fuzzy Modeling,
	30. dubna 22, 70103 Ostrava,
	Czech Republic}
\email{jan.boronski@osu.cz}
\author{Ji\v{r}\'{\i} Kupka}
\address[J. Kupka]{National Supercomputing Centre IT4Innovations, Division of the University of Ostrava,
	Institute for Research and Applications of Fuzzy Modeling,
	30. dubna 22, 70103 Ostrava,
	Czech Republic}
	\email{jiri.kupka@osu.cz}
\author{Piotr Oprocha}
\address[P. Oprocha]{AGH University of Science and Technology, Faculty of Applied
	Mathematics, al.
	Mickiewicza 30, 30-059 Krak\'ow, Poland
	-- and --
	National Supercomputing Centre IT4Innovations, Division of the University of Ostrava,
	Institute for Research and Applications of Fuzzy Modeling,
	30. dubna 22, 70103 Ostrava,
	Czech Republic}
\email{oprocha@agh.edu.pl}

\begin{abstract}
We show that every (invertible, or noninvertible) minimal Cantor system embeds in $\mathbb{R}$ with vanishing derivative everywhere.
We also study relations between local shrinking and periodic points.
\end{abstract}
\maketitle
\section{Introduction}
A \emph{Cantor set} is a 0-dimensional compact metric space without isolated points, and \emph{Cantor system} is a dynamical system on the Cantor set. A minimal system is the one that has all orbits dense. The present paper is concerned with the following question.
\begin{que}\label{q1}
Can every minimal Cantor system be embedded into $\mathbb{R}$ with vanishing derivative everywhere?
\end{que}
A particular instance of that question was raised by Samuel Petite at the {\em Workshop on Aperiodic Patterns in Crystals, Numbers and Symbols} that took place in Lorentz Center in June of 2017, who asked if expansive minimal Cantor systems have this property. It was conjectured during that meeting that the expansive systems lack such a property, because some kind of expanding must take place in these systems. In contrast, we answer Question \ref{q1} in the affirmative.
\begin{thm}\label{Th:LRS}
	Let $(C,T)$ be a minimal Cantor system.	
	Then there exists
	an embedding $\pi \colon C\to \R$ such that the map $f:\pi (C)\to\pi (C)$ given by $f=\pi \circ T\circ \pi^{-1}$
	has derivative $0$ everywhere.
\end{thm}
 Note that minimal Cantor systems occur quite naturaly as subsystems of interval maps (see e.g. \cite{blokh2}). Since differentiable maps defined on perfect subsets of $[0,1]$ can be extended to differentiable maps of the interval (see e.g. \cite{Jar}), this gives in particular the following realization theorem.
\begin{thm}\label{thm:1}
Every minimal Cantor system $(C,T)$ can be realized as a minimal subsystem of a differentiable system $([0,1],f)$ such that $f'|_C\equiv 0$.
\end{thm}
For expansive systems, which are usually connected with hyperbolic dynamics, this result is counter-intuitive. This is because in such systems any two distinct points have to eventually separate to some positive distance.
On the other hand one may think that because of vanishing derivative, points have to be attracted to each other for arbitrarily long time. This is not true however, because
only points sufficiently close to each other are attracted (say, at distance $\eps_x$ from $x$), while in practice $\eps_{T^n(x)}$ decrease much faster
than the distance $d(T^n(x),T^n(y))$, and so eventually these points can separate.

There are more reasons for which the above result seems surprising. By the Margulis-Ruelle inequality the topological entropy of a piecewise Lipschitz differentiable map $f$, with an invariant measure $\mu$, is bounded from above by the integral over the support of $\mu$ of the Lyapunov characteristic of $f$. In the case of derivative zero, all Lyapunov exponents, and as a result
 Lyapunov characteristic of $f$ are all equal to $0$.
Therefore it is  natural to expect that vanishing derivative on an invariant set will imply zero entropy on that set. Such an intuition was supported by the zero entropy examples in \cite{BKO}. However the main result in this paper shows that no such connection exists. Let us recall plethora of surprising minimal examples that can be observed on Cantor set. All starts with theorem of Jewett and Krieger, and constructive examples in a series of papers of Grillenberger, cf. \cite{Gril}. Further results made classes of systems and their dynamical properties more specific.
For example, there exist topologically weakly mixing minimal Cantor systems with arbitrary entropy in $[0,\infty]$, because it was later proven that every aperiodic ergodic system has a topologically mixing strictly ergodic model \cite{Lehrer}. The class of minimal Cantor systems also contains non-uniquely ergodic systems, since any possible simplex of measures is realizable by a Cantor minimal system. This can be done even in the class of Toeplitz minimal systems (see \cite{Downar}, cf. \cite{D1}), which are almost 1-1 extensions of odometers. Furthermore, there are examples with several measures of maximal entropy (see \cite{Downar}, cf. \cite{DS}). At the opposite end, there are examples of pure point systems (i.e. measurable conjugate to rotation over a group), e.g. see \cite{DL}. As we can see, most of known dynamics can be observed in minimal systems on the Cantor set, and this can hardly be connected with derivative zero.
Note that if $f$ is $C^1$ on $\R$ and $f(P)\subset P$ for some perfect compact subset $P\subset \R$, then there is $x\in P$ with  $|f'(x)|\geq 1$ (e.g. see Lemma~3.3 in \cite{CJREA}).
For systems with positive entropy it is also a consequence of Margulis-Ruelle inequality mentioned above, so in this case the map is not even Lipschitz continuous. This give raise to the following question.
\begin{que}
Can the map $f$ in Theorem~\ref{thm:1} be additionally required to be $\alpha$-H\"older continuous for some $0<\alpha<1$?
\end{que}

Theorem \ref{Th:LRS} generalizes an earlier result of the present authors, who in \cite{BKO} showed that all odometers can be embedded into $\mathbb{R}$ with vanishing derivative everywhere. The first result of that kind, using very different methods, was achieved by Ciesielski and Jasi\'nski in \cite{CJ} for the 2-adic odometer. Surjective dynamical systems with vanishing derivative everywhere constitute a subclass of a larger family of systems that are \textit{locally radially shrinking} (\textit{l.r.s.} for short); i.e. surjective dynamical systems $(X,f)$ such that
\begin{itemize}
	\item[(LRS)] for every $x\in X$ there exists an $\epsilon_x>0$ such that $d(x,y)<\epsilon_x$ implies $d(f(x),f(y))<d(x,y)$ for all $y\neq x$.
\end{itemize}
It was shown in \cite{CJ} that each infinite l.r.s. system contains an infinite minimal subsystem, and that for any $n$ the set of $n$-periodic points
is finite. The natural question arises, whether the set of all periodic points can be infinite. Our next result answers this question in the affirmative, with a class of $0$-dimensional systems. These systems are obtained from Cantor systems by adding some trajectories, but are not Cantor systems themselves (the set $Z$ below has isolated points).
\begin{thm}\label{infinite}
	Every Cantor minimal l.r.s. system $(C,T)$ can be extended to a non-transitive l.r.s. system $(Z,F)$, such that the set of periods of $(Z,F)$ is unbounded.
\end{thm}
Let us finish this section with the following observations. First, Question \ref{q1} is not particularly interesting for transitive dynamical systems that are not periodic point free. Suppose that $(X,T)$ is a l.r.s. system and $p\in X$ is a periodic point. Then it is easy to see that there exists an open set $U\ni p$ such that $\omega_T(x)=\Orb(p)$ for every $x\in U$. This immediately implies that if a transitive l.r.s. system has periodic point $p$ then $X=\Orb(p)$. Second, in \cite{BKO} there were constructed a nonminimal (hence infinite) transitive l.r.s. system, and an l.r.s. system with an attractor-repellor pair. In a sense the construction of the attractor-repellor pair from \cite{BKO} is the simplest possible and cannot be much improved  by the observation below. Namely, if we have an attractor-repellor pair in an l.r.s. system, then the attractor can be even a fixed point (see e.g. Theorem 5.3 in \cite{BKO}) but the repellor must be an infinite set (must be a Cantor minimal set or larger).
Precisely, let $A_n=\{y : T^i(y)=x \text{ for some } i\geq n\}\neq \emptyset$ and let $\alpha(x)$ be a generalized $\alpha$-limit set of $x$; i.e.
	$\alpha(x)=\bigcap_n \overline{A_n}$. Note that $\alpha(x)$ is closed, invariant and nonempty.
	
\begin{prop}
	Suppose that $(X,T)$ is an l.r.s. system and $x\in X$. If $x$ is not periodic, then $\alpha(x)$ contains no periodic orbits. In particular, $\alpha(x)$ contains an infinite minimal subsystem.
\end{prop}
\begin{proof}
 We claim that if $q\in \alpha(x)$ then $q$ is not a periodic point. Suppose to the contrary, that it is not the case and $q$ is periodic. Clearly $q\neq x$ as $x$ is not periodic. For simplicity we can assume that $q$ is fixed.
Therefore,  by the fact that $(X,T)$ is l.r.s., there is an open set $V\ni q$
	such that $x\not\in V$ and $T(V)\subset V$. Denote $A_n=\{y : T^i(y)=x \text{ for some } i\geq n\}$. Then $A_n\cap V\neq \emptyset$ for any $n$, and therefore $x\in V$ which is a contradiction proving the claim. On the other hand, $\alpha(x)$ contains a minimal set which must be infinite, since it is not a periodic orbit.
\end{proof}

\section{Preliminaries}
In this section we recall notions from the theory of graph covers, that we are going to use in the proof of Theorem \ref{Th:LRS}. For more on graph covers, and various recent applications see e.g. \cite{Good,GambaudoMartens06,Shim2,Shi,ShimomuraMinimal}.
\subsection{Graph covers}
A \textit{graph} is defined as a pair $G=(V,E)$ of finite sets, where \emph{vertices} are represented by elements of the set $V$ and \emph{edges} of the graph $G$ are represented by elements of $E\subseteq V\times V$. We say that the graph $G$ is \emph{edge surjective} if every vertex of $G$ has an incoming and outgoing edge. This means that for every $v\in V$ there exist $u,w\in V$ such that $(u,v),(v,w) \in E$. For two graphs $(V_1,E_1)$ and  $(V_2,E_2)$ a map $\phi \colon V_1\to V_2$ defines a \textit{homomorphism} if $\phi$ is edge-preserving, i.e.  for each $(u,v)\in E_1$ we obtain $(\phi(u),\phi(v))\in E_2$. For simplicity of notation, homeomorphism of graphs is denoted by $\phi \colon (V_1,E_1)\to (V_2,E_2)$. A graph homomorphism $\phi$ is called \textit{bidirectional} if $(u,v),(u,v')\in E_1$ implies $\phi(v)=\phi(v')$ and
$(w,u),(w',u)\in E_1$ implies $\phi(w)=\phi(w')$. A bidirectional homomorphism between edge-surjective graphs is called a \emph{bd-cover}. Given a sequence $\G=\seq{\phi_i}_{i=0}^\infty$  of bd-covers $\phi_i \colon (V_{i+1},E_{i+1})\to (V_i,E_i)$ we denote by $V_\G$ the space given by
$$
V_\G=\varprojlim(V_i,\phi_i)=\{ x\in \Pi_{i=0}^\infty V_i : \phi_i(x_{i+1})=x_i \text{ for all }i\geq 0\}
,$$
i.e. $V_\G$ is the inverse limit of the sequence $\G$. We set $\phi_{m,n}=\phi_n\circ \phi_{n+1}\circ \ldots \circ \phi_{m-1}$,  and by $\phi_{\infty,n}$ we denote the projection from $V_\G$ onto $V_n$.
We let
$$
E_\G=\{e\in V_\G\times V_\G : e_i\in E_i \text{ for every }i=0,1,2,\dots \} .
$$
$V_i$ is endowed with the discrete topology and the space $V_\G=\prod_{i=0}^\infty V_i$
is endowed with the product topology, that is equivalent to the metric topology given by $d(x,y)=0$ when $x=y$ and
$d(x,y)=2^{-k}$
provided $x\neq y$ and $k=\min \{i : x_i\neq y_i\}$.

Given a graph $G$ by $V(G)$ and $E(G)$ we denote respectively the set of vertices and edges of $G$. A \emph{path} $v_1,v_2,v_3, \ldots ,v_{l-1},v_{l}$ in a graph $G$ is a finite subgraph of $G$ given by edges
$$(v_1,v_2),(v_2,v_3), \ldots, (v_{l-2},v_{l-1}),(v_{l-1},v_{l}).$$
Given a path $\eta$, by $E(\eta)$ we denote the set of edges
in $\eta$. A \emph{cycle} in $G$ is a path that starts and ends at the same vertex. Given cycles $c_1,\ldots, c_n$ that start at the same vertex $v$, by $a_1 c_1+\ldots +a_n c_n$ we denote
the cycle starting and ending at the vertex $v$, obtained by passing $a_1$ times cycle $c_1$, then $a_2$ times cycle $c_2$, and so on. We shall need the following result from \cite[Lemma 3.5]{Shim2}.
\begin{lem}\label{lem:Tg}
	Let $\G=\seq{\phi_i}$ be a sequence of bd-covers $\phi_i \colon (V_{i+1},E_{i+1})\to (V_i,E_i)$.
	Then $V_\G$ is a zero-dimensional compact metric space and the relation $E_\G$ defines a homeomorphism.
\end{lem}
\subsection{Coverings of Gambaudo-Martens type}\label{subs:GMcover}
 A description of all (not necessarily invertible) minimal zero-dimensional dynamical systems in terms of inverse limits of graph covers of particular type was given by Gambaudo and Martens \cite{GambaudoMartens06}. Following Shimomura  \cite{ShimomuraMinimal} we refer to them as coverings of \emph{Gambaudo-Martens type} (GM-coverings for short). Let $\mathcal G$ be a sequence of graph covers $(\varphi_i)_{i=0}^\infty$, where $\varphi_i : (V_i,E_i) \to (V_{i-1}, E_{i-1})$ is a graph homomorphism between graphs $G_i$ and $G_{i-1}$. The graph $G_0$ is a graph consisting of one special vertex $v_0$ and one special edge $(v_0,v_0)$. In the following definition each graph $G_i$ consists of a special vertex $v_{i,0}$, a special edge $(v_{i,0}, v_{i,0})$ and a finite number of cycles $c_{i,j}$ which start and end in $v_{i,0}$ and which have the property that once they meet at one vertex, they coincide until they reach the vertex $v_{i,0}$. Strictly speaking, a sequence of graph covers $(\varphi_{i})_{i=1}^\infty$ is a \emph{GM-covering} if the following conditions are satisfied:
\begin{enumerate}
\item the cycle $c_{i,j}$ can be written as $v_{i,0} = v_{i,j,0}, v_{i,j,1}, v_{i,j,2}, \ldots , v_{i,j,l(i,j) }= v_{i,0}$ with the length $l( i,j)\geq1$,
\item $\bigcup_{j=1}^{r_i} E(c_{i,j})= E(G_i)$, where $r_i$ denotes the number of cycles $c_{i,j}$ in $G_i$,
\item $v_{i,j,k} = v_{i,\bar j, \bar k}$ for some $j, \bar j \geq 1$ implies $v_{i,j,k+m} = v_{i,\bar j, \bar k+m}$ for every $m = 1,2,3, \ldots , l(i,j)-k$,
\item $\varphi_i (v_{i,0}) = v_{i-1,0}$ for every $i\geq 1$,
\item $\varphi_i (v_{i,j,1}) = v_{i-1, 1, 1}$ for $i\geq 1$ and $j= 1,2,\ldots , r_i$.
\end{enumerate}
A GM-covering is called \emph{simple} if, for $i\geq 1$, there exists $m>i$ such that
$$
E(\varphi_{m,i} (c_{m,j})) = E(G_i),
$$
for each cycle $c_{m,j}$ in $G_m$. By telescoping, we can assume even that $m=i+1$, i.e.
$$
E(\varphi_{i+1} (c_{i+1,j}) ) = E (G_{i}),
$$
for every $i\geq 1$ and every cycle $c_{i+1,j}$ in $G_{i+1}$.

In what follows we will need the following fact from \cite{GambaudoMartens06}.

\begin{lem}\label{lem:GambaudoMartens} A zero-dimensional dynamical system is minimal if and only if it can be represented as the inverse limit of a simple GM-covering.
\end{lem}

\section{Proof of Theorem \ref{Th:LRS}}\label{sec:odometers}
\begin{proof}[Proof of Theorem \ref{Th:LRS}]
By Lemma \ref{lem:GambaudoMartens}, the minimal Cantor system $(C,T)$ can be represented as the inverse limit of GM-cover $(\varphi_i)$. In the proof below we follow notation of GM-covering from Section \ref{subs:GMcover}. Let $s_m$ denote the number of vertices in the graph $G_m$ and $r_m$ denote the number of cycles defining the graph $G_m$.
By telescoping we can assume that $s_{m+1} >4s_m^2$ for every $m$.

We denote by $W^\prime_m$ the set of vertices $\{v_{m, j, l(m-1,1)}\}$. And by $W_m$ we denote the subset of $W^\prime_m$ such that each path between $w\in W_m$ and $v_{m,0}$ contains no other point from $W^\prime_m$ but each cycle in $G_m$ contains at least one of them.
By condition (5) of the definition of GM covering $\varphi_{m}(c_{m,j})$ covers $V(G_{m-1})$ and starts by the cycle $c_{m-1,1}$.
Consequently after $l(m-1,1)$ steps in $\varphi_{m}(c_{m,j})$ one gets back to the special vertex $v_{m-1,0}$.
This immediately implies that $\varphi_{m}(w)=v_{m-1,0}$ for each $w\in W_m$.

Since each path from $W_m$ to $W_m$ without any inner vertex in $W_m$ does not have cycles,
we can assign index to each vertex of $G_m$, say $V_m=\{w^{m}_1,\ldots, w^m_{s_m}\}$ in such a way that if we take any  path $v_0,\ldots, v_k$ in $G_m$ without vertices in $W_m$ and $v_0=w^m_i$, $v_k=w^m_j$
then $i<j$. We may additionally assume that minimal indexes are in $W_m$, that is if $w^m_i\in W_m$ and $w^m_j\not\in W_m$ then $i<j$.

For technical reasons we put $s_{0}=1$,
put 
$a_0=\frac{1}{6}$, $b_0=2^{-2s^2_0}a_0$ and $A_i^{(1)}=[i,i+3a_0]$ for $i=1,\ldots, s_1$.
For each $m>0$ 
we will define the function $\psi_m\colon V_m\to (0,1)$ by putting
$$
\psi_m(w)=2^{-2s_{m+1}^2-is_{m+1}}a_{m-1},\qquad \text{ when }w=w^m_i.
$$
Note that at this point $\psi_1$ is defined.

For $i=1,\ldots, s_1$ let $D_i^{(1)}\subset A_i^{(1)}$ be an interval of length $\psi_1(w^1_i)$ placed in the middle of $A_i^{(1)}$, that is $ A_i^{(1)}\setminus D_i^{(1)}$
has two connected components which are intervals of equal length.

Suppose sets $D_i^{(k)}$, $A_i^{(k)}$ are defined for $k=1,\ldots, m$ and $i=1,\ldots, s_k$.
We put $a_{m}=\max_{i} \diam A_i^{(m)}/3$ and $b_{m}=2^{-s_{{m+1}}^2}a_{m}$.

Divide each $D_i^{(m)}$ into $s_{m+1}$ intervals of equal length and disjoint interiors.
For each vertex $w=w^{m+1}_i\in V_{m+1}$ we assign uniquely one of those intervals in $D_{r}^{(m)}$, where $r$ is such that $\varphi_{m+1}(w)=w^m_r$. Since for every $v\in V_{m}$ we have $|\varphi_{m+1}^{-1}(v)|\leq s_{m+1}$
we are able to assign pairwise distinct intervals to vertices. We name interval assigned to $w^{m+1}_i$ as $A^{(m+1)}_i$ and
let $D_i^{(m+1)}\subset A_i^{(m+1)}$ be an interval of length $\psi_{m+1}(w_i^{m+1})$ placed in the middle of $A_i^{(m+1)}$, that is $ A_i^{(m+1)}\setminus D_i^{(m+1)}$
has two connected components which are intervals of equal length. Such an inclusion is possible, because
\begin{eqnarray}
\diam A_i^{(m+1)}&=&\diam D_i^{(m)}/s_{m+1}\geq 2^{-2s_{m+1}^2-s_{m+1}^2-s_{m+1}}a_{m-1}\nonumber \\
&\geq& 2^{-4s_{m+1}^2}a_{m-1}>  2^{-2s_{m+2}}a_{m-1}\geq 3 \cdot 2^{-2s_{m+2}}a_{m}\label{eq:hole:size}
\end{eqnarray}
and $\psi_{m+1}(w_i^{m+1})< 2^{-2s_{m+2}}a_{m}$.
It also implies that the diameter of each connected component of $ A_i^{(m+1)}\setminus D_i^{(m+1)}$ is greater than $D_i^{(m+1)}$.

For any $z\in V_\G$ the intersection $\bigcap_{m} D_{z_m}^{(m)}$ is a single point $x_z$, because
$$
\diam D_{z_m}^{(m)}\leq 2^{-s_m}a_{m}\leq 2^{-m},
$$
for any $m$. Furthermore, if $y,z\in V_\G$ and $y_m\neq z_m$ then $x_z\in D_{y_m}^{(m)}\subset \Int A_{y_m}^{(m)}$
and $x_y\in D_{z_m}^{(m)}\subset \Int A_{z_m}^{(m)}$ and these sets are disjoint.
This shows that the map $\pi \colon V_\G \ni z\mapsto x_z\in \R$ is well defined, continuous and injective.
Denote $X=\pi(V_\G)$. Then $\pi \colon V_\G \to X$ is a homeomorphism and $X$ is a Cantor set.
Define $f=\pi \circ T_\G\circ \pi^{-1}$. We are going to show that $f'(x)=0$ for every $x\in X$.

Fix any $x\in X$, any sequence $x_n\to x$ and let $z=\pi^{-1}(x)$ and $z^{(n)}=\pi^{-1}(x_n)$. We may assume that $x\neq x_n$ for every $n$,
hence there exists a sequence $j_n$ such that $z^{(n)}_{j_n}=z_{j_n}$ and $z^{(n)}_{j_n+1}\neq z_{j_n+1}$. Observe that by the definition of the function $\varphi_m$ and the set $W_m$ (recall that $\varphi_{m-1}(W_m)=v_{m-1,0}$) there exists at most one $n$ such that $z_n\in W_n$.
Therefore, if $n$ is sufficiently large then there are $r_n$ and $t_n\geq r_n+1$ such that
\begin{eqnarray*}
	\diam D_{z_n}^{(n)}&=&2^{-2s_{n+1}^2-r_n s_{n+1}}a_{n-1},\\
	\diam D_{T_\G(z)_n}^{(n)}&=&2^{-2s_{n+1}^2-t_n s_{n+1}}a_{n-1}.
\end{eqnarray*}
If $p,q\in X$ are such that $p\in  D_{r}^{(n)}$ and $q\in D_{s}^{(n)}$, and $r\neq s$, then by \eqref{eq:hole:size} we have
$$
|p-q| \geq \frac{\diam A^{(n+1)}_r - \diam D^{(n+1)}_r} {2} > \diam A^{(n+1)}_r/3. 
$$
By the above estimates we obtain
\begin{eqnarray*}
\frac{|f(z)-f(z_n)|}{|z-z_n|}&\leq& \frac{\diam D^{(j_n)}_{(T_\G(z))_{j_n}}}{\diam A_{z_{j_n+1}}^{(j_n+1)}/3}\leq \frac{3\cdot {2^{-s_{j_n+1}}} \diam D_{z_{j_n}}^{(j_n)}}{{\diam A_{z_{j_n+1}}^{(j_n+1)}}}\\
&\leq &\frac{3\cdot s_{j_n +1}\cdot \diam A_{z_{j_n+1 }}^{(j_n+1)}}{{2^{s_{j_n+1}}} {\diam A_{z_{j_n+1}}^{(j_n+1)}}} = \frac {3\cdot s_{j_n +1}}{2^{s_{j_n+1}}}
\longrightarrow 0.
\end{eqnarray*}
Indeed $f'(z)=0$ completing the proof.
\end{proof}

\section{Proof of Theorem \ref{infinite}}
\begin{proof}[Proof of Theorem \ref{infinite}]
The proof is based on a construction from \cite[Theorem 5.1]{BKO}. Take any point $z\in C$.
Since $(C,T)$ is a Cantor l.r.s. system, there exists a nested sequence of closed-open neighborhoods $U_n$ of the point $z$, such that $d(T(z),T(x))<d(z,x)$ for every $x\in U_1$, $x\neq z$, and $\bigcap_n U_n=\{z\}$. Since $(C,T)$ acts minimally and aperiodically, going to a subsequence (i.e. removing some of the sets $U_n$) if necessary, we can find
an increasing sequence of integers $0\leq s_n<k_n$ such that:
\begin{enumerate}[(i)]
\item	$T^{-i}(z)\not\in U_n$, for $s_n< i < k_n$, and
\item $T^{-k_n}(z),T^{-s_n}(z)\in U_n$,
\item $T^{-i}(z)\not\in U_{n+1}$, for $0< i \leq k_n$,
\item $d(z,T^{-s_n}(z))>d(z,T^{-k_n}(z))$.
\end{enumerate}
Observe that the set $U_{n}\setminus U_{n+1}$ is closed
for every $n$, and consequently the following number
\begin{equation}
0<a_n < \min_{y\in U_n\setminus U_{n+1}}d(z,y)-d(T(z),T(y))\label{eq:lrcest}
\end{equation}
exists. We may assume that $a_{n+1}<a_n/2$ and $a_1<1/2$.

For each $n\geq 1$ and $j=0,1,\ldots, k_n-s_n-1$ let us denote
$$
z^{n}_j=(T^{j-k_n+1}(z),-1+\frac{k_n-s_n-1-j}{2(k_{n}-s_{n}-1)}(a_n+a_{n+1})+a_{n}/2).
$$

Next, we denote $Z=(C\times \{-1\})\cup \bigcup_{n\geq 1, 0\leq j<k_n-s_n}\{z^n_j\}$ and we endow $Z$ with metric $\rho((x,q),(x',q'))=d(x,x')+|q-q'|$.
Note that $\lim_{n\to \infty}a_n=0$ and so $Z$ is compact.
Define $F(x,-1)=(T(x),-1)$ for every $x\in C$ and $F(z^n_j)=z^n_{j+1 (\text{mod } (k_n-s_n))}$.
It is not hard to see that $F$ is continuous.
It remains to show that $F$ is an l.r.s. map. The points $z_n^j$ are isolated in $Z$, so $F$ is l.r.s. at each of them.
Now fix any $y\in C\setminus \{z\}$. There is a closed-open set $V_y\subset C\setminus \{z\}$ such that $d(x,y)>d(T(x),T(y))$ for any $x\in V_y$. There is $m$ such that
$U_m\cap V_y=\emptyset$ and so $V_y\times [-1,-1+a_m/2]$ does not contain any of the points $z^n_{k_n-s_n-1}$.
But then for $(x,a)\in V_y\times [-1,-1+a_m/2]$, we have $F(x,a)=(T(x),b)$ where $a+1>b+1$ and so
\begin{eqnarray*}
\rho(F(y,-1),F(x,a))&=&d(T(y),T(x))+b+1\\
&<&d(y,x)+a+1=\rho((y,-1),(x,a)).
\end{eqnarray*}
The same calculation holds for $(z,-1)$ in pair with $z^n_j$, provided that $j<k_n-s_n-1$, and $z^n_j\in U_1\times \R$.
Finally, observe that by \eqref{eq:lrcest} we obtain
\begin{eqnarray*}
	\rho(F(z,-1),F(z_{k_n-s_n-1}^n))&=&\rho(F(z,-1),z_0^n)=d(T(z),T(T^{-k_n}(z)))+a_n+\frac{a_{n+1}}{2}\\
	&<&d(z,T^{-k_n}(z))+\frac{a_{n+1}}{2}<d(z,T^{-s_n}(z))+\frac{a_{n}}{2}\\
	&=&\rho((z,-1),z^n_{k_n-s_n-1}).
\end{eqnarray*}
The proof that $(Z,F)$ is an l.r.s. system is completed. Since we have infinitely many periodic points, the set of periods is unbounded by Lemma~12 in \cite{CJ}.
\end{proof}
\section*{Acknowledgements}
This work was supported by the subsidy for institutional development IRP201824 ”Complex topological structures” from University of Ostrava. J. Boro\'nski was supported by National Science Centre, Poland (NCN), grant no. 2015/19/D/ST1/01184.

\end{document}